\documentclass[12pt,a4paper,oneside]{article}
\usepackage{amsmath,amstext,amssymb,amsopn,amsthm,color,textcomp}
\usepackage{xcolor}

\theoremstyle{plain}
\newtheorem{thm}{Theorem}

\newtheorem{lem}[thm]{Lemma}

\newtheorem{cor}[thm]{Corollary}

\theoremstyle{definition}

\theoremstyle{remark}
\newtheorem*{rem*}{Remark}

\newcommand{\R}{\mathbb{R}}

\newcommand{\E}{\mathbb{E}}

\newcommand{\C}{\mathbb{C}}
\newcommand{\D}{\mathbb{D}}
\renewcommand{\H}{\mathbb{H}}

\renewcommand{\leq}{\leqslant}
\renewcommand{\le}{\leq}
\renewcommand{\geq}{\geqslant}
\renewcommand{\ge}{\geq}

\DeclareMathOperator{\re}{Re}
\DeclareMathOperator{\im}{Im}

\newcommand{\laplace}{\mathcal{L}}

\newcommand{\pref}[1]{(\ref{#1})}

\def\({\left(}
\def\){\right)}
\def\[{\left[}
\def\]{\right]}
\def\<{\langle}
\def\>{\rangle}

\title{Hitting half-spaces or spheres by the Ornstein-Uhlenbeck type diffusions
\footnotetext{2010 MS Classification:
                   Primary 60J45; Secondary 60G15, 60G40.
   {\it Key words and phrases}: harmonic measure, Ornstein-Uhlenbeck diffusion, Girsanov theorem, hyperbolic spaces, Poisson kernel.\\
   Research supported by Polish Ministry of Science and Higher Eduction grant N N201 3731 36 and l'Agence Nationale de la
Recherche grant no. ANR-09-BLAN-0084-01.} 
}

\author{Tomasz Byczkowski, Jakub Chorowski, Piotr Graczyk, Jacek Ma\l{}ecki}
\begin{document}
\maketitle

\abstract{The purpose of the paper is to provide a general method for computing hitting distributions 
of some regular subsets $D$ for Ornstein-Uhlenbeck type operators of the form $\frac{1}{2}\Delta + F\cdot\nabla$,
with $F$ bounded and orthogonal to the boundary of $D$.
As an important application we obtain integral representations of the Poisson kernel for a half-space and balls for hyperbolic Brownian motion and for the classical Ornstein-Uhlenbeck process.
The method developed in the paper is based on stochastic calculus and on skew product representation of multidimensional Brownian motion and yields more complete results as those based on Feynmann-Kac
technique.
}
\section{Introduction}

A detailed knowledge of the hitting distribution (equivalently: harmonic measure) of a domain for a diffusion with the given generator $\cal{A}$ is fundamental for solving
many potential-theoretic problems, e.g. Dirichlet boundary problem for a domain or Harnack inequality or even boundary Harnack inequality for harmonic functions with respect to $\cal{A}$.

In the paper is we compute hitting distributions of some subsets $D$ for operators of the form $\frac{1}{2}\Delta + F\cdot\nabla$ on subsets of $\R^{n}$. 
It is worth to point out that even in the case of the classical Ornstein-Uhlenbeck diffusion the explicit formulas for half-spaces or balls were obtained only quite 
recently (see \cite{AliliPatie:2005} and \cite{GraczykJakubowski:2008}).
 Although the inspiration for our work comes from the paper \cite{CranstonZhao:1987}, where the potential theory for bounded sets $D$ and  the operators
\begin{equation} \label{O-U-basic}
\frac{1}{2}\Delta + F\cdot\nabla
\end{equation}
was established, the purpose as well as most of technical tools 
are here different: instead of setting up a general theory, we focus on providing explicit formulas for  hitting distributions for some important operators of the above type and sets $D$. The importance of explicit formulas is highlited e.g. in the recent papers \cite{BMR3:2011} and \cite{HM:2011}, where precise asymptotics for Poisson kernel for Bessel diffusions were obtained.
Throughout the paper we assume that the vector field $F$ in \pref{O-U-basic} is bounded and orthogonal to the boundary of $D$. 
The method developed here is based on stochastic calculus and 
Girsanov's theorem and consists in computing various integral functionals of Brownian motion and representing them in terms of special functions. 

We provide a closed formulas for 
 the density function of hitting ditribution, i.e. Poisson kernel of a half-space or a ball for the  hyperbolic Brownian motion or for the classical Ornstein-Uhlenbeck  process. 
The importance of the hyperbolic Brownian motion stems from the fact that it is the canonical diffusion on hyperbolic spaces; it has also some important applications in the risk theory in financial mathematics, see \cite{D:1990} and \cite{Yor:1992}.
Explicit integral representations 
are crucial in obtaining   estimates of Poisson kernel and also of the Green function \cite{BGS:2007},  \cite{BMZ:2010}.
In these papers, the  main tool was the Feynman-Kac formula, applied  to describe the distribution of a stopped multiplicative functional.
The present approach, based on methods related to Girsanov's theorem 
enable us to obtain representation formulas for the Poisson kernel, different from those mentioned above.
The advantage of this approach is seen in Theorem 4, where we obtain the precise asymptotics of the Poisson kernel
for large values of parameters. Another, worth mentioning result is the Theorem 6, where we provide adequate representation 
of Poisson kernel of a ball. Also the formula for the Poisson kernel of a ball for the classical Ornstein-Uhlenbeck 
diffusion is more complete than obtained in \cite{GraczykJakubowski:2008} (as a series representation only).

The paper is organized as follows. In Section 2 we provide the general framework for the next sections. 
Throughout the paper we assume that in \pref{O-U-basic} we deal with potential vector field $F$ on $D$, orthogonal to the boundary. 
Under this assumption, with the aid of stochastic calculus and Girsanov's theorem, we establish a general
formula for the harmonic measure of the set $D$ (Theorem 2).

In Section 3 we provide a closed formula for the Poisson kernel $P(x_n,y)$  of half-space for the hyperbolic 
Brownian motion on the real hyperbolic space $\H^n$ (Theorem 3) and provide an asymptotic formula for $P(x_n,y)$ (Theorem 4).
In Section 4 we provide an integral representation of Poisson kernel for centered balls for hyperbolic Brownian motion on the ball model $\D^n$.  We remark here that a similar representation  from the paper \cite{BM:2007} depends on additional conjectures on the zeros
of some hypergeometric functions which so far remained unsettled. The important tool here, as in the next section, is the skew-product representation of the n-dimensional Brownian motion.
In Section 5 we provide an integral representation for Poisson kernel of a ball for the classical Ornstein-Uhlenbeck process
(Theorem 7).
In Appendix we collect some useful information on Bessel functions, hypergeometric and Legendre functions and on skew-product of n-dimensional Brownian motion.

\section{Change of measure due to Girsanov theorem}
\label{sec:harmonic}
\textbf{Notation:}
For $n>2$ we denote by $\R^n$  the $n$-dimensional Euclidean space, $\<x,y\>$ denotes the standard inner product of $x, y\in \R^n$ and by $|x|$ we denote the Euclidean length of a vector $x\in\R^n$. 
The ball with center at zero and the radius $r$ is written as $B_r=\{x\in \R^n: |x|<r\}$; its boundary, which is the $(n-1)$-dimensional sphere, is denoted by $S_r^{n-1} = \{x\in \R^n: |x|=r\}$ and the spherical measure on $S_r^{n-1}$ is denoted by $\sigma_r^{n-1}$.
Furthermore, we write $f(x)\sim g(x)$, when $\lim_{x\to b}{f(x)}/{g(x)}=1$, as $x \to b$. If for two functions $f$ and $g$ there exist constants $c_1$, $c_2$ such that $c_1<{f(x)}/{g(x)}<c_2$ for every $x\in D$ we will write shortly $f\approx g, x\in D.$

Throughout the paper $D$ will stand for a domain $D\subset \R^n$ with a smooth, connected boundary $\partial D$ and $F$ will be a bounded vector field which is defined on an open Lipschitz set $U$ containing $D$. We assume that $F$ is continuously differentiable up to the boundary of U and continuously vanishes on the boundary of $U$. We further assume that $F=\nabla V$ on $D$ for a scalar valued function $V$ and call the function $V$ a potential (and $F$ a potential vector field on $D$). We set $F=V\equiv 0$ on the complement of $U$. 
We say that  the vector field $F$ is orthogonal to the boundary $\partial D$ of the set $D$ if for every differentiable curve
$\Gamma: [0,1)\to \partial D$ we have $F(\Gamma(s)) \cdot\Gamma '(s)=0$ for every $s\in[0,1)$. 
\begin{lem}
Under the above assumptions, if the potential vector field $F$ is orthogonal to the boundary of the set $D$ then the potential function $V$ determined, up to a constant,  by the equation $$\nabla V(x)=F(x)$$
is constant on the boundary $\partial D$. 
\end{lem}

\begin{proof}
Set $x_0\in\partial D$. The potential $V$ is described by the curve integral
\begin{eqnarray*}
  V(x)=\int_\gamma F(r)  dr + V(x_0),
\end{eqnarray*}  
where $\gamma$ is an arbitrary continuously differentiable path beginning at $x_0$ and ending at $x$. For $x\in \partial D$ we  choose $\gamma$ to follow the boundary of the set D i.e. $\gamma: [0,1]\to\partial D$, $\gamma(0)=x_0$, $\gamma(1)=x$. Then
$$V(x)-V(x_0)=\int_0^1 F(\gamma(s))\cdot \gamma '(s)ds=0\,.$$
Since $\partial D$ is connected, we obtain the conclusion.
\end{proof}
Throughout the paper we work within the framework of the canonical representation of the process, i.e. our basic probability space is the space
of all continuous $\R^n$-valued functions defined on $[0,\infty)$ with  appropriate $\sigma$-fields (see \cite{IW}).
The standard n-dimensional Brownian motion is denoted by $W(t)=(W_1(t),\ldots,W_n(t))$.

Define the process $X$ by the SDE
\begin{equation} 
\label{Xdiffusion}
d X(t)=dW(t)+F(X(t))dt\,,
\end{equation}
under the conditions specified above. Then 
X is a local diffusion on $U$ with the generator $L=\frac{1}{2}\Delta +F(x)\cdot \nabla.$ 
Since the field $F$ is bounded, 
$X$ can be defined as a local semimartingale
(see, e.g. \cite{IW}). Let  $\tau$ be the first exit time of the trajectory from the set $D$.
The harmonic measure $w^x$ on $\partial D$, is defined as the distribution of $X(\tau)$ under 
the distribution $P^x$ of the process $X$ starting at $x \in D$.
We define a local martingale $M$ by the formula

\begin{eqnarray}
\label{eq:Mt:def}
  M(t) = \int_0^t F(W(s))\cdot dW(s).
\end{eqnarray}
Its quadratic variation is then given by the formula
\begin{eqnarray*}
 \left<M\right>(t) = \int_0^t |F(W(s))|^2ds.
\end{eqnarray*}
We further define the basic object of our study, namely
\begin{equation} 
\label{kernel}
N(t) = \exp\( M(t) - \frac{1}{2}\left<M\right>(t)\)\,.
\end{equation}
We now provide the basic formula for the harmonic measure of the process defined by \pref{Xdiffusion} under some additional conditions.
\begin{thm}\label{Girsanov}
Under conditions stated above, suppose that $X$ is the process defined by the system \pref{Xdiffusion}. 
Assume additionally  

(i) The vector field $F$ is potential and orthogonal to the boundary of the domain $D$,

(ii) For every $t>0$
\begin{equation} \label{Girsanovzalozenie}
E^x[\exp(\left\langle M\right\rangle(t\wedge\tau))] < \infty\,, 
\end{equation}

(iii)
\begin{equation} \label{Girsanovuniform}
\{N(t\wedge\tau)\}_{t>0}\quad {\text{is uniformly integrable}}\,.
\end{equation}

Then for $x\in D$ the harmonic measure $w^x$ has the density function expressed by the formula
\begin{equation*}
w^x(dz)=e^{V(\partial D)-V(x)}E^x\left[\exp\left(-\frac{1}{2}\int_0^\tau \left(|\nabla V(W(s))|^2+
\Delta V(W(s))
\right)ds\right)\/;W(\tau)\in dz\right],
\end{equation*}
where $V$ is the potential function of the field $F$ and $V(\partial D)$ is its value on the boundary $\partial D$. 
\end{thm}
\begin{proof}
According to \pref{eq:Mt:def} the process
\begin{equation*}
F(W(t))\cdot dW(t) - \frac{1}{2} |F(W(t))|^2 dt = dM(t) - \frac{1}{2}d\left<M\right>(t)
\end{equation*}
is a local semimartingale. Writing, as in \pref{kernel}
\begin{equation*}
N(t) = \exp(M(t) - \frac{1}{2}\left<M\right>(t))
\end{equation*}
we obtain, as an application of It\^o's formula, that $N(t)$ is a local martingale. 
If we define the measure $Q^x$ by 
\begin{equation*}
\frac{dQ^x}{dP^x}|_{{\cal{F}}_{t\wedge\tau}} = N(t\wedge \tau)\,,
\end{equation*}
then, as a consequence of Girsanov's theorem,  $(W,Q^x)$ and $(X,P^x)$ are different descriptions of the same process, up to time $\tau$, see \cite{CranstonZhao:1987}.
Consequently, for a continuous bounded function $f$ defined on $\R^n$  we obtain
\begin{equation*}
E^x f(X(t\wedge\tau))=E^x(N(t\wedge\tau); f(W(t\wedge\tau)))
\end{equation*}
Now, the condition \pref{Girsanovuniform} shows that the expression on the right-hand side converges to
$E^x[N(\tau); f(W(\tau))]$, as $t \to\infty$.  The left-hand side converges to $E^x f(X(\tau))$, by the continuity of 
the process $X$.
This indicates that indeed $w^x$ has density given by
\begin{eqnarray}
\label{G}
w^x(dz)=\E^x[N(\tau)\,;W(\tau)\in dz]
\end{eqnarray}
We now provide a further description of the function $w^x$. Recall that $F$ is the potential of the vector field $B$. Define
$$Z(t)=V(W(t)).$$
Applying the It\^o formula we see that 
\begin{eqnarray*}
Z(t)-Z(0)&=&\int_0^t \nabla V(W(s))\cdot dW(s)+\frac{1}{2}\int_0^t 
\Delta V(W(s))ds=\\
&=&M(t)-\frac{1}{2}\left<M\right>(t) +\frac{1}{2}\left<M\right>(t)+\frac{1}{2}\int_0^t 
\Delta V(W(s))ds.
\end{eqnarray*}
Consequently
\begin{eqnarray}
\label{N}
N(t)&=&\exp(M(t)-\frac{1}{2}\left<M\right>(t))\nonumber\\
&=&\exp\left(Z(t)-Z(0)-\frac{1}{2}\int_0^t\left[|\nabla V(W(s))|^2+
\Delta V(W(s))\right]ds\right).
\end{eqnarray}
Remark that $Y(\tau)=V(W(\tau))=V(\partial D)$ since the vector field F is orthogonal to the boundary of the set $D$. Hence, stopping at $t\wedge\tau$ and taking expectation  we get, when $t\to\infty$  
\begin{equation*}
w^x(dz)=
e^{V(\partial D)-V(x)}\E^x\left[\exp\left(-\frac{1}{2}\int_0^\tau \left[|\nabla V(W(s))|^2+
\Delta V(W(s))\right]ds\right);W(\tau) \in dz\right]\/.
\end{equation*}
This, together with \pref{G}, finishes the proof.

\end{proof}

\section{Harmonic measure of hyperbolic horocycle in $\H^n$}
\label{Ch:Hn}
For every $a>0$ we define $H_a=\{ x\in \R^n:\; x_n>a \}$. In this section we consider the harmonic measure $\omega_a^x$ of the set $H_a$ for the operator
\begin{eqnarray*}
\Delta_{LB} =  x_n^2 \sum_{i=1}^{n} \dfrac{\partial^2}{\partial x_i^2}- (n-2)x_n\dfrac{\partial}{\partial x_n}\/,\quad n\geq 2\/.
\end{eqnarray*}
The motivations for studying this operator comes from hyperbolic geometry. More precisely, this operator is the Laplace-Beltrami operator associated with Riemannian metric in the half-space model ${\H}^{n}$ of the real $n$-dimensional hyperbolic space. From geometric point of view, the set $H_a$ is an interior of the hyperbolic horocycle $\partial H_a = \{ x\in \R^n:\; x_n=a \}$. Let $(B_i(t))_{i=1...n}$ be an n-dimensional Brownian
motion on $\R^n$ with the generator $\frac{d^2}{dx^2}$ (and not $\frac{1}{2}\frac{d^2}{dx^2}$) i.e.  the variance $E^0B_i^2(t) = 2t$.  Then the
Brownian motion on ${\H}^n$, $Y=(Y_i)_{i=1...n}$, can be described by
the following system of stochastic differential equations
\begin{equation} \label{LB}
  \left\{
    \begin{array}{ccc}
      dY_1(t) & = & Y_n(t)dB_1(t) \\
      dY_2(t) & = & Y_n(t)dB_2(t) \\
      . & . &.  \\
      dY_n(t) & = & Y_n(t)dB_n(t) - (n-2)Y_n(t)dt.
    \end{array} \right.
\end{equation}
By the It\^o formula one verifies that the generator of the solution
of this system is $\Delta_{LB}$.
The Laplace-Beltrami operator can be rewritten in the form $\Delta_{LB}=2x_n^2 L_1$, where
\begin{equation} \label{LB-OU}
   L_1 = \frac{1}{2}\Delta+F_1(x)\cdot \nabla\/,
\end{equation}
with $F_1(x) = (0,\ldots,0,(2-n)/(2x_n))$. 
Now, we make change of time. Namely, we write
\begin{equation*}
A(u)= \int_0^u Y_{n}^2(s)\,ds\,,
\end{equation*}
and 
\begin{equation*}
\sigma (t) = \inf\{u>0; A(u)>t\}\,.  
\end{equation*}
If we now write 
\begin{equation*}
{\tilde{B}}_k(t) = \int_0^{\sigma(t/2)} Y_n(s)\,dB_k(s)\,,\quad k=1,...,n
\end{equation*}
then ${\tilde{B}}_k$ are martingales with mutual variations $<{\tilde{B}}_k,{\tilde{B}}_l>(t)=\delta(k,l)\,t$, $k,\,l = 1,...n$ so 
${\tilde{B}}=({\tilde{B}}_k)$ is the standard n-dimensional Brownian motion.
Substituting 
\begin{equation*}
{\tilde{Y}}_k(t) =  Y_k(\sigma(t/2))\,,\quad k=1,...,n
\end{equation*}
we obtain that  \pref{LB} transforms into the following the system of SDE
\begin{equation} \label{OU}
  \left\{
    \begin{array}{ccc}
      d{\tilde{Y}} _1(t) & = & d{\tilde{B}}_1(t) \\
      d{\tilde{Y}} _2(t) & = & d{\tilde{B}}_2(t) \\
      . & . &.  \\
      d{\tilde{Y}}_n(t) & = & Z_n(t)d{\tilde{B}}_n(t) - (n-2)\frac{dt}{2Z_n(t)}\,.
    \end{array} \right.
\end{equation}
Again, by It\^o's formula, we see that the operator $L_1$ is the generator of the process ${\tilde{Y}}=({\tilde{Y}}_k)$.
Since the operation of change of time does not affect the exit place, the harmonic measures of the operators
$\Delta_{LB}$ and $L_1$ are the same.
 
The potential of the vector field $F_1$ is given by $V_1(x) = (2-n)\ln(x_n)/2$. Moreover, it is easy to check that the vector field $F_1$ is orthogonal to the boundary of $H_a$. Using (\ref{eq:Mt:def}) and (\ref{N}), we obtain 
\begin{eqnarray*}
   M(t) &=& \frac{2-n}{2}\int_0^t \frac{dW_n(s)}{W_n(s)}\/,\quad \<M\>(t) = \(\frac{n-2}{2}\)^2\int_0^t \frac{ds}{W_n^2(s)}\/,\\
   N(t) &=& \(\frac{W_n(0)}{W_n(t)}\)^{(n-2)/2}\exp\(-\frac{n(n-2)}{8}\int_0^t \frac{ds}{W_n^2(s)}\)\/,
\end{eqnarray*}
where $W(t) = (W_1(t),\ldots,W_n(t))$ denotes the standard Brownian motion in $\R^n$ starting from $W(0) = x$.
If we put $\tau_a = \inf\{t>0: W(t)\notin H_a\}$ we obtain 
\begin{eqnarray*}
   \E^x\exp(\<M\>(t\wedge \tau_a)) &=& \E^x\exp\(\frac{(n-2)^2}{4}\int_0^{t\wedge \tau_a}\frac{ds}{W_n^2(s)}\)\\
   &\leq& \E^x\exp\(\frac{(n-2)^2}{4}\int_0^{t\wedge \tau_a}\frac{ds}{a^2}\)<\infty\/.
\end{eqnarray*}
Moreover, we have
\begin{eqnarray*}
   N(t\wedge \tau_a)  \leq \(\frac{W_n(0)}{W_n(t\wedge \tau_a)}\)^{(n-2)/2}\leq \(\frac{x_n}{a}\)^{(n-2)/2}\/,\quad t\geq  0.
\end{eqnarray*}
Now, the results of Theorem~\ref{Girsanov} imply that
\begin{eqnarray}
  \label{eq:Hn:OmegaForm}
  \omega^x_a(dy) = \left(\frac{x_n}{a}\right)^{\frac{n-2}{2}}\E^x\left[\exp\left(-\frac{n(n-2)}{8}\int_0^{\tau_a} \frac{ds}{W_n(s)^2}\right);W(\tau_a)\in dy\right]\/.
\end{eqnarray}
The above-given formula enable us to find the density function $P_a(x,y)$, $x\in H_a$, $y\in\partial H_a$, of the measure $\omega^x_a(dy)$ with respect to the Lebesgue measure on $\partial H_a$. The scaling property of $n$-dimensional Brownian motion implies the scaling property for Poisson kernels $P_a(x,y)=a^{1-n}P_1(x/a,y/a)$, $x\in H_a$, $y\in \partial H_a$. Moreover, Brownian motion $W(t)$ and the set $H_1$ are invariant under translations $(\tilde{x},x_n)\to (\tilde{x}+b,x_n+b)$, where $b\in \R^{n-1}$. Consequently, $P_1(x,y) = P_1((0,x_n),(\tilde{y}-\tilde{x},1))$ for every $x\in H_1$ and $y\in \partial H_1$. We will use these properties in further considerations to simplify the notation. We will denote by $P(x_n,y) = P_1((0,x_n),(y,1))$, where $y\in \R^{n-1}$ and $\tau = \tau_1$.

\begin{thm}
\label{thm:Hn:PoissonForm}
For every $x_n>1$ and $y\in \R^{n-1}$ we have
\begin{equation}
\label{eq:Hn:PoissonForm}
P(x_n,y)= \frac{1}{2^{\nu-1}\pi^{\nu+1}}\frac{x_n^\nu}{|y|^{\nu-1}}
\int_0^\infty \frac{J_{\nu}(t)Y_{\nu}(tx_n)-J_{\nu}(tx_n)Y_{\nu}(t)}
{J_{\nu}^2(t)+Y_{\nu}^2(t)}\,t^{\nu}K_{\nu-1}\left({t|y|}\right)dt\/,
\end{equation} 
where $\nu=(n-1)/2$.
\end{thm}
\begin{proof}
Observe that the integral appearing in (\ref{eq:Hn:OmegaForm}) as well as the hitting time $\tau$ depend only on the last coordinate of the Brownian motion $W(t)=(\tilde{W}(t),W_n(t))$. Since the processes $\tilde{W}(t)$ and $W_n(t)$ are independent we obtain
\begin{eqnarray}
\nonumber
   \omega_1^x(dy) &=& x_n^{\frac{n-2}{2}}\int_0^\infty \E^0[\tilde{W}(s)\in dy]\,\E^{x_n}\left[\exp\left(-\frac{n(n-2)}{8}\int_0^{\tau} \frac{ds}{W_n(s)^2}\right);\tau \in ds\right] \\
   \label{eq:Hn:HM:positive}
   &=& x_n^{\frac{n-2}{2}}\left(\int_0^\infty \frac{\exp(-|y|^2/(2s))}{(2\pi s)^{(n-1)/2}}\,\mu_{x_n}(ds)\right)dy \/,
\end{eqnarray}
where
\begin{eqnarray*}
   \mu_{x_n}(ds) = \E^{x_n}\left[\exp\left(-\frac{n(n-2)}{8}\int_0^{\tau} \frac{ds}{W_n(s)^2}\right);\tau \in ds\right]\/.
\end{eqnarray*}
The Laplace transform of $\mu_{x_n}$ is given by
\begin{eqnarray*}
   \laplace \mu_{x_n}(w) &=&  \E^{x_n}\exp\left(-\frac{n(n-2)}{8}\int_0^{\tau} \frac{ds}{W_n(s)^2}-w\int_0^\tau ds\right) = E^{x_n}e_q(\tau)\/,\quad w\geq 0\/,
\end{eqnarray*}
where $q(x)=-{n(n-2)}/{(8x^2)}-w$. The function $\varphi(x_n) = E^{x_n}e_q(\tau)$ is a gauge function for an appropriate Schr\"o{}dinger operator based on generator of $W_n(t)$. Consequently, $\varphi$ is a bounded solution of the equation
\begin{eqnarray*}
   \frac12 \varphi''(x) - \left(\frac{n(n-2)}{8x^2}+w\right)\phi(x)=0\/,\quad x\geq 1
\end{eqnarray*} 
such that $\varphi(1) = 1$. Making substitution $\sqrt{x}\psi(x\sqrt{2w})=\varphi(x)$ we reduce the above-given equation to
\begin{eqnarray*}
2x^2 w\psi ''(x\sqrt{2w})+x\sqrt{2w}\psi '(x\sqrt{2w})-\left(\frac{(n-1)^2}{4}+2wx^2\right)\psi(x\sqrt{2w})=0,
\end{eqnarray*}
which is the modified Bessel equation (\ref{eq:mBessel:Eq}) with $\nu=\frac{n-1}{2}$. Taking into account the general form of solutions of (\ref{eq:mBessel:Eq}), the boundary condition and boundedness of $\varphi$ we arrive at
\begin{eqnarray}
   \label{eq:Hn:muLaplaceForm}
   \laplace \mu_{x_n}(w) &=& \sqrt{x_n}\,\frac{K_{\nu}(x_n\sqrt{2w})}{K_\nu(\sqrt{2w})}\/,\quad w\geq 0\/.
\end{eqnarray}

Since square root can be extended to a holomorphic function on $\C\setminus(-\infty,0]$ and the modified Bessel function $K_\nu$ has no zeros in the positive half-plane $\re w\geq 0$, the Laplace transform $\laplace \mu_{x_n}(w)$ can also be extended to analytic function on $\C\setminus(-\infty,0]$. Moreover, using the asymptotic expansion (\ref{eq:mBessel:Asym}) we obtain that 
\begin{eqnarray}
  \label{eq:Hn:muLaplaceEst}
  |\laplace \mu_{x_n}(w)| \leq \left|e^{-(x_n-1)\sqrt{2w}}\frac{1+E(x_n\sqrt{2w})}{1+E(\sqrt{2w})}\right| \leq 2\exp\left({-(x_n-1)\sqrt{2|w|}\cos\frac{\arg w}2 }\right)\/,
\end{eqnarray}
for every $w \in \C\setminus(-\infty,0]$ such that $|w|$ is large enough. Note that here $\arg w \in[-\pi,\pi]$. In particular, $\laplace \mu_{x_n}(w)$ is bounded for $|w|\geq 1$. These properties of $\laplace \mu_{x_n}$ and its analytic continuation guarantee that we can apply the inverse Laplace transform to (\ref{eq:Hn:muLaplaceForm}) (see \cite{Folland:1992}, Theorem 8.5, p. 267). More precisely, there exist a density function of $\mu_{x_n}$ with respect to Lebesgue measure on $(0,\infty)$ given by the inversion formula
\begin{eqnarray*}
   \mu_{x_n}(s) = \frac{1}{2\pi i} \lim_{r\to \infty} \int_{1-ir}^{1+ir} \laplace \mu_{x_n}(w) e^{sw}\,dw\/.
\end{eqnarray*}
To compute the limit we integrate the function $f_s(w) = \laplace \mu_{x_n}(w) e^{sw}$ over rectangular contour surrounding the branch-cut of $f_s$ which is negative real axis. Let $\Gamma$  be the positively oriented contour consisting of four horizontal segments $\gamma_1 = [-r+i/r,i/r]$, $\gamma_2 = [-r-i/r,-i/r]$, $\gamma_3= [-r+ir,1+ir]$, $\gamma_4 = [-r-ir,1-ir]$, three vertical segments $\gamma_5 = [-r+i/r,-r+ir]$, $\gamma_6 = [-r-i/r,-r-ir]$, $\gamma_7 = [1-ir,1+ir]$ and a semi-circle $\gamma_8 = \{|w|=1/r,\re w>0 \}$. The formula (\ref{eq:mBessel:AsymZero}) implies that $f_s$ is bounded for small $w$ such that $\re w>0$. Consequently, the integral over $\gamma_8$ tends to zero when $r\to \infty$. The boundedness of $\laplace \mu_{x_n}(w)$ for large $w$ implies that for $r\geq 1$ and every $s>0$ we have
\begin{eqnarray*}
  \left|\left(\int_{\gamma_5}+\int_{\gamma_6}\right)f_s(w)dw\right|\leq 2\sup_{|w|\geq 1}\laplace \mu_{x_n}(w)\, r e^{-rs}\to 0\/,
\end{eqnarray*}
as $r\to \infty$. Finally, using (\ref{eq:Hn:muLaplaceEst}), we obtain
\begin{eqnarray*}
 \left|\left(\int_{\gamma_3}+\int_{\gamma_4}\right)f_s(w)dw\right|\leq 4 \exp({-(x_n-1)\cos(3\pi/8)\sqrt{r}})\int_{-1}^\infty e^{-su}du \to 0
\end{eqnarray*}
as $r\to \infty$. The Cauchy Theorem together with previous considerations and (\ref{eq:mBessel:ImArg}) give
\begin{eqnarray*}
   \mu_{x_n}(s) &=& \frac{1}{2\pi i} \lim_{r\to \infty} \int_{1-ir}^{1+ir} \laplace \mu_{x_n}(w) e^{sw}\,dw
   = \frac{1}{2\pi i}\lim_{r\to\infty}\left(\int_{\gamma_1}+\int_{\gamma_2}\right) f_s(w)dw\\
   &=& \frac{\sqrt{x_n}}{2 \pi i} \int_{0}^\infty\left[\frac{K_{\nu}(-i\sqrt{2t}x_n)}{K_{\nu}(-i\sqrt{2t})}-\frac{K_{\nu}(i\sqrt{2t}x_n)}{K_{\nu}(i\sqrt{2t})}\right]e^{-st}dt\\
   &=& -\frac{\sqrt{x_n}}{\pi}\int_0^{\infty}\im\left(\frac{K_{\nu}(i\sqrt{2t}x_n)}{K_{\nu}(i\sqrt{2t})}
\right)e^{-st}dt\\
   &=& -\frac{\sqrt{x_n}}{\pi}\int_0^{\infty}\im\left(\frac{K_{\nu}(itx_n)}{K_{\nu}(it)}
\right)te^{-st^2/2}dt\\
&=& \frac{\sqrt{x_n}}{\pi}\int_0^\infty \frac{J_\nu(t)Y_\nu(tx_n)-J_\nu(tx_n)Y_\nu(t)}{J_\nu^2(t)+Y_\nu^2(t)}te^{-st^2/2}dt\/.
\end{eqnarray*}
From (\ref{eq:mBessel:asymptotyka0}) and (\ref{eq:mBessel:asymptotykainfty}) is easy to see that $J_\nu^2(t)+Y_\nu^2(t)\sim t^{-1}\vee t^{-2\nu}$ and
\begin{eqnarray*}
\left|\frac{J_\nu(t)Y_\nu(tx_n)-J_\nu(tx_n)Y_\nu(t)}{J_\nu^2(t)+Y_\nu^2(t)}\right|\le C \left(1+t\right)\/,\quad t>0
\end{eqnarray*}
for some constant $C=C(x_n)>0$. With the use of (\ref{eq:mBessel:Integral}) we verify
\begin{eqnarray*}
\int_0^\infty \frac{|J_\nu(t)Y_\nu(tx_n)-J_\nu(tx_n)Y_\nu(t)|}{J_\nu^2(t)+Y_\nu^2(t)}t \left(\int_0^\infty \frac{e^{-|y|^2/(2s)}e^{-st^2/2}ds}{(2\pi s)^{(n-1)/2}}\right)dt\\
\le C \int_0^\infty (t+1)t^{\nu} K_{\nu-1}(t|y|)dt.
\end{eqnarray*}
The last integral is finite by (\ref{eq:mBessel:Asym}) and (\ref{eq:mBessel:AsymZero}). Consequently, by the Fubini's theorem we obtain
\begin{eqnarray*}
  P(x_n,y) &=& \frac{x_n^{\nu}}{\pi}\int_0^\infty \frac{J_\nu(t)Y_\nu(tx_n)-J_\nu(tx_n)Y_\nu(t)}{J_\nu^2(t)+Y_\nu^2(t)}t \left(\int_0^\infty \frac{e^{-|y|^2/(2s)}e^{-st^2/2}ds}{(2\pi s)^{(n-1)/2}}\right)dt\\
  &=& \frac{1}{2^{\nu-1}\pi^{\nu+1}}\frac{x_n^\nu}{|y|^{\nu-1}} \int_0^\infty \frac{J_{\nu}(t)Y_{\nu}(tx_n)-J_{\nu}(tx_n)Y_{\nu}(t)}
{J_{\nu}^2(t)+Y_{\nu}^2(t)}\,t^{\nu}K_{\nu-1}\left({t|y|}\right)dt\/.
\end{eqnarray*}
This ends the proof.
\end{proof}

The integral formula presented in Theorem~\ref{thm:Hn:PoissonForm} can be used to obtain the asymptotics of the Poisson kernel $P(x_n,y)$ as well as its sharp bounds for small $x_n$ and large $|y|$. Note that results given in the next theorem cover those obtained in \cite{BGS:2007} (see Theorem 4.9 and Theorem 4.10, compare also with Theorem 5.3 in \cite{BR:2006}). Moreover, the formula (\ref{eq:Hn:PoissonForm}) allows to omit very laborious and sophisticated computations used in \cite{BGS:2007} to examine the behavior of $P(x_n,y)$ when $|y|$ tends to infinity. Our approach is simpler and gives more general results.

\begin{thm}
\label{thm:Hn:PK:asymptotics}
For every $x_0\geq 1$ we have
\begin{eqnarray}
\label{b}
P(x_n,y)&\sim& \frac{\Gamma(n/2)}{2^{n-2}\pi^{n/2}}\sum_{k=0}^{n-2}x_0^k\cdot\frac{x_n-1}{|y|^{2n-2}}\/,\quad  x_n\to x_0,|y|\to\infty\/.
\end{eqnarray}
Moreover, for every $y_0> 0$ we have
\begin{eqnarray}
\label{c}
P(x_n,y)&\sim& \frac{c(y_0)}{\pi^{(n+3)/2}2^{(n-5)/2}}\cdot(x_n-1)\/,\quad x_n\to 1,|y|\to y_0\/,
\end{eqnarray}
where 
\begin{eqnarray*}
   c(y_0) = |y_0|^{\frac{1-n}{2}}\int_0^\infty \frac{s^{\nu}K_{\nu-1}(sy_0)ds}{J_\nu^2(s)+Y_\nu^2(s)}\/.
\end{eqnarray*}

\end{thm}
\begin{proof}
 Making a substitution  $t|y|=s$ in (\ref{eq:Hn:PoissonForm}) we can rewrite the Poisson kernel in the following way
\begin{eqnarray*}
\frac{P(x_n,y)|y|^{4\nu}}{x_n-1}=\frac{x_n^{\nu}}{\pi^{\nu+1} 2^{\nu-1}} \int_0^\infty g_\nu\left(x_n,\frac{s}{|y|}\right)s^{3\nu} K_{\nu-1}(s)ds\/,
\end{eqnarray*}
where
\begin{eqnarray*}
 g_\nu(x,t)=\frac{1}{t^{2\nu}(x-1)}\frac{J_\nu(t)Y_\nu(xt)-J_\nu(tx)Y_\nu(t)}{J_\nu^2(t)+Y_\nu^2(t)}\/,\quad x>1,t>0\/.
\end{eqnarray*}
Since $s^{3\nu} K_{\nu-1}(s)$ is integrable on $[0,\infty)$ and (\ref{oszacowaniafr}) gives boundedness of  $\left|g_\nu\left(x_n,\frac{s}{|y|}\right)\right|$ for $x_n<R$, we can apply Lebesgue dominated convergence theorem to get
\begin{eqnarray*}
\lim_{(x_n,|y|)\to(1,\infty)}\frac{P(x_n,y)|y|^{4\nu}}{x_n-1}
&=&\frac{1}{\pi^\nu 2^{3\nu-2}\Gamma^2(\nu)}\int_0^\infty s^{3\nu}K_{\nu-1}(s)ds
=\frac{2\Gamma(2\nu)\Gamma(\nu+1)}{\pi^\nu\Gamma^2(\nu)}
\end{eqnarray*}
and
\begin{eqnarray*}
\lim_{(x,|y|)\to(x_0,\infty)}\frac{P(x_n,y)|y|^{4\nu}}{x_n-1}&=&\frac{x_0}{\pi^{\nu+1}2^{\nu-1}}\int_0^\infty \frac{\pi (x_0^{2\nu}-1)s^{3\nu}K_{\nu-1}(s)ds}{2^{2\nu}(x_0-1)\Gamma(\nu)\Gamma(\nu+1)x_0^\nu}\\
&=&\frac{(x_0^{2\nu}-1)\Gamma(2\nu)}{\pi^\nu(x_0-1)\Gamma(\nu)x_0^{\nu-1}}=\frac{\left(\sum_{k=0}^{n-2}x_0^k\right)\Gamma(2\nu)}{\pi^\nu x_0^{\nu-1}\Gamma(\nu)}
\end{eqnarray*}
whenever $x_0>1$. Here we used formula (\ref{eq:JY:limit1}) from Lemma~\ref{granice} and relation (\ref{eq:mBessel:Integral2}). The duplication formula for gamma function gives (\ref{b}). In the same way, using (\ref{eq:JY:limit1}), we get
\begin{eqnarray*}
\lim_{(x_n,|y|)\to(1,y_0)}\frac{P(x_n,y)|y|^{4\nu}}{x-1}&=&\frac{1}{\pi^{\nu+1}2^{\nu-1}}\int_0^\infty \frac{2s^{\nu}|y_0|^{2\nu}K_{\nu-1}(s)ds}{\pi\left(J_\nu^2\left(\frac{s}{|y_0|}\right)+Y_\nu^2\left(\frac{s}{|y_0|}\right)\right)}.
\end{eqnarray*}
which, by substitution $s=ty_0$, proves (\ref{c}).
\end{proof}

As a consequence of Theorem~\ref{thm:Hn:PK:asymptotics} we obtain the following sharp bounds for the Poisson kernel for small $x_n$ and large $|y|$. Similar results have been obtained recently (see Theorem 11 in \cite{BMR3:2011}). The results are more general (there is no restriction for $x_n$ and $|y|$) however the methods of proof are much more complicated.
\begin{cor}
We have
\begin{eqnarray*}
P(x_n,y)\approx\frac{x_n-1}{|y|^{2n-2}},\quad 1< x_n\leq 2,|y|\geq 1\/.
\end{eqnarray*}
\end{cor}
\begin{proof}
Existence and positivity of the limits proved in Theorem~\ref{thm:Hn:PK:asymptotics} imply that for every $x\in[1,2]$ there exist $\varepsilon_x>0$ and $Y_x>0$ and strictly positive constants $c_1(x), c_2(x)$ such that
\begin{eqnarray*}
   c_1(x)\frac{x'-1}{|y|^{2n-2}}\leq {P(x',y)}\leq c_2(x)\frac{x'-1}{|y|^{2n-2}}
\end{eqnarray*}
for every $x'\in [1,2]$ and $y\in\R^{n-1}$ satisfying $|x-x'|<\varepsilon_x$ and $|y|>Y_x$. Since the family of intervals $(x-\varepsilon_x,x+\varepsilon_x)$ is an open cover of $[1,2]$, we can choose the finite subcover $\{(x_i-\varepsilon_{x_i},x_i+\varepsilon_{x_i}), i=1,\ldots m\}$. Putting $Y=\max\{Y_{x_i}: i=1,\ldots,m\}$, $c_1=\min\{c_1(x_i): i=1,\ldots,m\}$ and $c_2=\max\{c_2(x_i): i=1,\ldots,m\}$ we get
\begin{eqnarray}
  \label{eq:Hn:bounds}
   c_1\frac{x_n-1}{|y|^{2n-2}}\leq {P(x_n,y)}\leq c_2\frac{x_n-1}{|y|^{2n-2}}
\end{eqnarray}
for every $x_n\in [1,2]$ and $|y|>Y$. Observe that formula (\ref{eq:Hn:HM:positive}) implies positivity of $P(x_n,y)$. Moreover, by (\ref{c}), we get that the function $P(x_n,y)|y|^{2n-2}/(x_n-1)$ can be continuously extended to the strictly positive function on the compact set $[1,2]\times[1,Y]$ and consequently (\ref{eq:Hn:bounds}) is true also for $|y|\in[1,Y]$ (with possibly different constants $c_1$ and $c_2$). This ends the proof.
\end{proof}

\section{Harmonic measure of hyperbolic balls in $\D^n$}
In this section we consider the harmonic measure of the balls associated with the operator 
\begin{eqnarray}
   \label{HBall:generator:formula}
 \Delta_{LB}= \frac{(1-|x|^2)^2}{4} \sum_{i=1}^n \dfrac{\partial^2}{\partial x_i^2}+(n-2)\frac{1-|x|^2}{2}\sum_{i=1}^n x_i\dfrac{\partial}{\partial x_i}\/.
\end{eqnarray}
This operator appears naturally as the Laplace-Beltrami operator on the ball model of the real hyperbolic space $\D^n$ (see \cite{BM:2007} for more details). 

In particular, if $B=(B_k)$ is the standard n-dimensional Brownian motion then 
 the following system of SDE:
\begin{equation*}
    \label{sde01}
    \frac{dY_k(t)}{1-|Y(t)|^2} = dB_k(t)+2(n-2)Y_k(t)dt, \quad
    k=1,\ldots,n \/.
\end{equation*}
describes a diffusion with values in the real hyperbolic space $\D^n$ , with the generator $2\Delta_{LB}$.
As in the case of half-space model, we perform a change of time defined by:
\begin{equation*}
A(u)= \int_0^u (1-|X(s)|^2)\,ds\,, \quad {\text{and}} \quad \sigma(t) = \inf\{u; A(u)>t\}\,.
\end{equation*}
Then the process defined by
\begin{equation*}
{\tilde{Y}}_k(t) = Y_k(\sigma(t)) 
\end{equation*}
is the diffusion with values in $\D^n$ with the generator $L_2$.

As in the case of $\H^n$ in Section 4,
the harmonic measures of the operators
$\Delta_{LB}$ and $L_2$ are the same.

We consider now the harmonic measure $\omega_r^x$ of a ball $B_r = \{x\in \R^n: |x|<r\}$, $r<1$, supported on the boundary of $B_r$ which is the sphere $S_r^{n-1}$ of radius $r$. We denote by $P_r(x,y)$ the Poisson kernel of $B_r$, i.e. the density of the measure $\omega^x_r$ with respect to the $(n-1)$-dimensional spherical measure $\sigma_r^{n-1}$. As in the previous section we can write the Laplace-Beltrami operator in the form $\frac{(1-|x|^2)^2}2 L_2$, where 
\begin{eqnarray*}
   L_2 = \frac{1}{2}\Delta + F_2(x)\cdot\nabla\/,
\end{eqnarray*} 
with $F_2(x) = \frac{n-2}{1-|x|^2}(x_1,\ldots,x_n)$.

The positivity of the factor $(1-|x|^2)^2/2$ implies that the harmonic functions on the ball $B_r$ for the operators (\ref{HBall:generator:formula}) and $L_2$ are exactly the same and consequently the harmonic measures coincide.

Moreover, the vector field $F_2$ is orthogonal to the sphere $S_r^{n-1}$ and its potential function is $V_2(x)=\frac{2-n}{2}\ln(1-|x|^2)$. We denote by $\tau_r = \inf\{t>0; W(t)\notin B_r\}$ the first exit time of Brownian motion $W(t)$ from a ball $B_r$. Note that $\tau_r$ depends only on the Euclidean norm of $W$.

The martingale $M$, related to the vector field $F_2$, and its quadratic variation are 
\begin{eqnarray*}
M(t)&=&(n-2)\int_0^t\frac{W_i(s)dW_i(s)}{1-|W(s)|^2}\/,\quad 
\<M\>(t)= (n-2)^2\int_0^t\frac{|W(s)|^2ds}{(1-|W(s)|^2)^2}.
\end{eqnarray*}
Observe that the condition \pref{Girsanovzalozenie} is fulfilled in this case, since
\begin{eqnarray*}
\E^x\exp\left[(n-2)^2\int_0^{t\wedge\tau_r} \frac{|W(s)|^2 \/ds}{(1-|W(s)|^2)^2}\right]
&\le&\E^x\exp\left[(n-2)^2 \frac{r^2(t\wedge\tau_r)}{(1-r^2)^2}\right]<\infty\,.
\end{eqnarray*}
By (\ref{N}), the kernel $N(t)$ is of the form
\begin{eqnarray*}
N(t)&=& \left(\frac{1-|W(0)|^2}{1-|W(t)|^2}\right)^{\frac{n-2}{2}}
\exp\left(-\frac{n(n-2)}{2}\int_0^t\frac{ds}{(1-|W(s)|^2)^2}\right)\/,
\end{eqnarray*}
and it is now evident that $\{N(t\wedge\tau_r)\}_{t>0}$ is uniformly bounded in $t$ so the condition \pref{Girsanovuniform}
holds. Applying Theorem \ref{Girsanov} we obtain
\begin{eqnarray}
\label{HBall:Girsanov:formula}
w^x_r(dy) = \left(\frac{1-|x|^2}{1-r^2}\right)^{\frac{n-2}{2}}\E^x\left[\exp \left(-\frac{n(n-2)}{2}\int_0^{\tau_r} \frac{ds}{(1-|W(s)|^2)^2}\right)\/; W(\tau_r)\in dy\right]\/,
\end{eqnarray}

From now on we assume that $x\neq 0$. For $x=0$, from the rotational invariance of the Laplace-Beltrami operator, we easily obtain that $\omega_r^x$ is just $\sigma_r^{n-1}/\sigma_r^{n-1}(S_r^{n-1})$. Recall the skew-product representation of the Brownian motion $$W(t) = R^{(\nu)}(t)\Theta(A^{(\nu)}(t)),$$ where $R^{(\nu)}$ is the Bessel process with index $\nu=\frac{n}{2}-1$ staring from $|x|$ and $\Theta$ is spherical Brownian motion on $S^{n-1}_1$ independent from $R^{(\nu)}$ (see Appendix). Using the fact that $\tau_r$ depends only on $R^{(\nu)}$ we get that $W(\tau_r) = R^{(\nu)}(\tau_r)\Theta(A^{(\nu)}(\tau_r))$, where $\Theta$ is independent from $R^{(\nu)}(\tau_r)$ and $A^{(\nu)}(\tau_r)$. Applying this decomposition to formula (\ref{HBall:Girsanov:formula}) we get
\begin{equation*}
   \omega_r^x(dy) = \left(\frac{1-|x|^2}{1-r^2}\right)^{\frac{n-2}{2}}\int_0^\infty P^{\frac{x}{|x|}}(\Theta_t \in dy)\E^{|x|}\left[\exp\(\int_0^{\tau_r} q(R_s^{(\nu)})ds\);A^{(\nu)}(\tau_r)\in dt\right]\/,
\end{equation*}
where $q(y) = -\frac{n(n-2)}{2(1-y^2)^2}$. Rotational invariance of spherical Brownian motion implies that the harmonic measure $\omega_r^x$ is axially symmetric with axis $x$. As a consequence we get that its density $P_r(x,y)$ depends only on the cosine of angle between starting point $x$ and the point $y$, i.e. 
\begin{eqnarray*}
P_r(x,y)= \tilde{P}_r\(x,\frac{\<x,y\>}{|x||y|}\)\/.
\end{eqnarray*}
If we consider the sets of the form $A = \{\eta \in S_1^{n-1}: \frac{\<x,\eta\>}{|x|}\in (a,b)\}$, where $-1<a<b<1$ we get (for a definition of the process $S$ see Appendix)
\begin{eqnarray*}
   P^{\frac{x}{|x|}}(\Theta_t \in A) = P^1(S_t \in (a,b))
    = \int_a^b p_t^S(1,z)m(dz) = 2 \int_a^b p_t^S(1,z)(1-z^2)^{(n-3)/2}dz\/,
\end{eqnarray*}
where $p_t^S$ is defined in (\ref{S:ptxy:definition}).
From the other side, using the spherical coordinates we obtain
\begin{eqnarray*}
   \omega_r^x(rA) &=& \int_{rA} P_r(x,y)d\sigma_r^{n-1}(y) = r^{n-1} \sigma_r^{n-2}(S_1^{n-2}) \int_{\cos\phi\in (a,b)} \tilde{P}_r(x,\cos\phi)\sin^{n-2}\phi \,d\phi\\
   &=& \frac{n\pi^{\frac{n-1}{2}} r^{n-1} }{\Gamma\(\frac{n+1}{2}\)} \int_a^b \tilde{P}_r(x,z)(1-z^2)^{(n-3)/2} \,dz\/.
\end{eqnarray*}
Comparing both sides we get the following formula for the Poisson kernel $P_r(x,y)$
\begin{eqnarray}
\label{Pformula}
   P_r(x,y) = \frac{\Gamma\(2\frac{n+1}{2}\)}{\pi^{\frac{n-1}{2}} n r^{n-1} }\left(\frac{1-|x|^2}{1-r^2}\right)^{\frac{n-2}{2}}\int_0^\infty p_t^{S}\(1,\frac{\<x,y\>}{|x||y|}\)\mu_{|x|}(dt)\/,
\end{eqnarray}
where
\begin{eqnarray*}
\mu_y(dt) = \E^{y}\left[\exp\(\int_0^{\tau_r} q(R^{(\nu)}(s))ds\);A^{(\nu)}(\tau_r)\in dt\right]\/,\quad y\in (0,r]\/.
\end{eqnarray*}
The formula for $p_t^S$ can be computed from the appropriate formula for the transition density function for $\Theta$, which is given in terms of spherical harmonics and that approach leads to the series representation for $P_r(x,y)$ presented in \cite{BM:2007}. However, we want to compute the Laplace transform of $p_t^{S}$ which is so called $\lambda$-Green function of the process $S$ 
\begin{eqnarray*}
   G_\lambda(x,1) = \int_0^\infty e^{-\lambda t}p_t^S(1,x)dt\/,\quad x\in (-1,1),
\end{eqnarray*}
and we do it directly. From the general theory (see for example \cite{BorodinSalminen:2002} Chapter II for short resume) a function $G_\lambda$ is described by solutions of the second-order differential equation
\begin{eqnarray}
   \label{S:Green:equation}
   \frac{1-x^2}{2}u''(x) - \frac{n-1}{2}\,x\,u'(x) = \lambda u(x)\/,\quad x\in (-1,1)\/.
\end{eqnarray}
Note that the expression on the left-hand side is just $\mathcal{G}u(x)$, where $\mathcal{G}$ is just the generator of $S$ described in (\ref{S:generator}). More precisely, we have
\begin{eqnarray*}
   G_\lambda(x,1) = \frac{\varphi_\lambda(1)\psi_\lambda(x)}{w_\lambda}\/,\quad x\in (-1,1)\/,
\end{eqnarray*}
where $\varphi_\lambda$ is a decreasing and $\psi_\lambda$ is an increasing solution of (\ref{S:Green:equation}) such that $\varphi_\lambda^{-}(1)=0$ and $\psi_\lambda^{+}(-1)=0$.  The boundary conditions for the derivatives follows from the fact that non-singular points $-1$ and $1$ are reflecting. Here $f^{+}$ and $f^{-}$ denote the right and left derivative with respect to the speed function $s(x)$. The Wronskian $w_\lambda$ is given by
\begin{eqnarray*}
   w_\lambda = \psi_\lambda^{-}(x)\varphi_\lambda(x)-\psi_\lambda(x)\varphi^{-}_\lambda(x)
\end{eqnarray*}
and it does not depend on $x$. Putting $x=1$ in the above-given formula and using the boundary conditions we obtain that 
\begin{eqnarray*}
   G_\lambda(x,1) = \frac{\varphi_\lambda(1)\psi_\lambda(x)}{\psi_\lambda^{-}(1)\varphi_\lambda(1)-\psi_\lambda(1)\varphi^{-}_\lambda(1)} = \frac{\psi_\lambda(x)}{\psi^{-}_\lambda(1)}\/,\quad x\in (-1,1)\/.
\end{eqnarray*}
This implies that $G_\lambda(x,1)$ is uniquely described as a solution of (\ref{S:Green:equation}) such that $u^{+}(-1) = 0$ and $u^{-}(1) = 1$. Making a substitution $u(x) = f\(z\)$ with $z= \frac{1+x}{2}$ in the equation (\ref{S:Green:equation}) we reduce it to the following hypergeometric equation
\begin{eqnarray*}
z(1-z)f''(z)+\(\frac{n-1}{2}-(n-1)z\)f'(z)-2\lambda f(z) = 0\/.
\end{eqnarray*}
with $\alpha = \frac{n-2}{2}-A(\lambda)$, $\beta = \frac{n-2}{2} +A(\lambda)$, $\gamma = \frac{n-1}{2}$, where $A(\lambda) = \frac12 \sqrt{(n-2)^2-8\lambda}$.
Consider the function
\begin{eqnarray*}
  h_\lambda(x)&=& \,_2F_1\(\frac{n-2}{2}-A(\lambda),\frac{n-2}{2}+A(\lambda);\frac{n-1}{2};\frac{1+x}{2}\)\/.
  \end{eqnarray*}
The above-given computation implies that the function $h_\lambda$ is a solution of (\ref{S:Green:equation}). Using (\ref{Hyper:derivative}) and (\ref{Hyper:relation}) we compute the derivative of this function  with respect to the scale function $s(x)$ in the following way
\begin{eqnarray*}
  (1-x^2)^{\frac{n-1}{2}}\frac{d}{dx}h_\lambda(x) &=& (1-x^2)^{\frac{n-1}{2}}\frac{2\lambda}{n-1}  \,_2F_1\(\frac{n}{2}-A(\lambda),\frac{n}{2}+A(\lambda);\frac{n+1}{2};\frac{1+x}{2}\)\\
  &=& \frac{2^{\frac{n+1}{2}}\lambda }{n-1}(1+x)^{\frac{n-1}{2}}\,_2F_1\(\frac{1}{2}+A(\lambda),\frac{1}{2}-A(\lambda);\frac{n+1}{2};\frac{1+x}{2}\)\/.
\end{eqnarray*}
The first equality and the fact that the hypergeometric function $\,_2F_1$ is equal to $1$ at zero implies $h_\lambda^{+}(-1) = 0$. Using the second equality and (\ref{Hyper:value1}) we obtain
\begin{eqnarray*}
  h_\lambda^{-}(1) &=& \frac{2^{n}\lambda }{n-1}\,_2F_1\(\frac{1}{2}+A(\lambda),\frac{1}{2}-A(\lambda);\frac{n+1}{2};1\)\\
  &=&  \frac{2^{n-1}\lambda \Gamma^2\(\frac{n-1}{2}\)}{\Gamma\(\frac{n-2}{2}-A(\lambda)\)\Gamma\(\frac{n-2}{2}+A(\lambda)\)}\/.
\end{eqnarray*}
Moreover, using once again (\ref{Hyper:relation}) and the definition (\ref{Legendre:first:formula}) we can express function $h_\lambda$ in terms of the Legendre function of the first kind
\begin{eqnarray*}
   h_\lambda(x) &=& \(\frac{1-x}{2}\)^{\frac{3-n}{2}}\,_2F_1\(\frac{1}{2}-A(\lambda),\frac{1}{2}+A(\lambda);\frac{n-1}{2};\frac{1+x}{2}\)\\
   &=& \frac{(1-x^2)^{\frac{3-n}{4}}}{2^{\frac{3-n}{2}}}\Gamma\(\frac{n-1}{2}\)P_{A(\lambda)-\frac12}^{\frac{3-n}{2}}(-x)\/.
\end{eqnarray*}
Finally we have just obtained that 
\begin{eqnarray*}
   G_\lambda(x,1) &=& B_{n}(\lambda) (1-x^2)^{\frac{3-n}{4}}P_{A(\lambda)-\frac12}^{\frac{3-n}{2}}(-x)\/,
\end{eqnarray*}
where
\begin{eqnarray*}
  B_{n}(\lambda)  = \frac{\Gamma\(\frac{n-2}{2}-A(\lambda)\)\Gamma\(\frac{n-2}{2}+A(\lambda)\)}{2^{\frac{n+1}{2}}\lambda \Gamma\(\frac{n-1}{2}\)}\/.
\end{eqnarray*}

The second part of the formula (\ref{Pformula}) relates to the measure $\mu_y$. Observe that $\mu_y$ depends only on the Bessel process $R^{(\nu)}$, which is one-dimensional diffusion. For every $w\geq 0$ the Laplace transform $\mathcal{L}\mu_y(w)$ is 
\begin{eqnarray*}
\E^{y}\exp\left(-\frac{n(n-2)}{2}\int_0^{\tau_r} \frac{ds}{(1-(R^{(\nu)}(s))^2)^2}-w\int_0^{\tau_r} \frac{ds}{(R^{(\nu)}(s))^2}\right)
 &=& \E^{y}e_g(\tau_r)\/,
\end{eqnarray*}
where $g(y) = -\frac{n(n-2)}{2(1-y^2)^2}-\frac{w}{y^2}$. The function $\varphi(y) = \E^{y}e_g(\tau_r)$ is by definition the gauge function for
the Schr\"o{}dinger operator based on the generator of the process $R^{(\nu)}$ and the non-positive potential $g$. From the Feynman-Kac formula $\varphi$ is a solution of the Schr\"o{}dinger equation. Using (\ref{Bessel:generator}) we obtain that $\varphi$ is a bounded solution to the following second-order differential equation
\begin{eqnarray}
   \label{HBall:Schrodinger:eq}
   \frac{1}{2}\varphi''(y)+\frac{n-1}{2y}\varphi'(y)-\left(\frac{n(n-2)}{2(1-y^2)^2}+\frac{w}{y^2}\right)\varphi(y)=0\/,\quad y\in [0,r)\/, w\geq 0\/,
\end{eqnarray}
with the boundary condition $\varphi(r)=1$. Substituting $\varphi(y) = y^{1-n/2}\psi\left(\frac{1+y^2}{1-y^2}\right)$ we obtain
\begin{eqnarray*}
   \varphi'(y) &=& \left(1-\frac{n}{2}\right)y^{-n/2}\psi\left(\frac{1+y^2}{1-y^2}\right)+\frac{4y^{2-n/2}}{(1-y^2)^2}\psi'\left(\frac{1+y^2}{1-y^2}\right)\/,\\
   \varphi''(y) &=& \frac{n}{2}\left(\frac{n}{2}-1\right)y^{-1-n/2}\psi\left(\frac{1+y^2}{1-y^2}\right)+\frac{16y^{3-n/2}}{(1-y^2)^4}\psi''\left(\frac{1+y^2}{1-y^2}\right)+\\
   &&+\frac{4y^{1-n/2}}{(1-y^2)^3}(3-n+(n+1)y^2)\psi'\left(\frac{1+y^2}{1-y^2}\right)\/.
\end{eqnarray*}
Putting the above given formulas to the differential equation (\ref{HBall:Schrodinger:eq}) and dividing both sides by the factor $-\frac{2y^{1-n/2}}{(1-y^2)^2}$ give
\begin{eqnarray*} 
0&=&\frac{-4y^2}{(1-y^2)^2}\psi''\left(\frac{1+y^2}{1-y^2}\right)-\frac{2(1-y^4)}{(1-y^2)^2}\psi'\left(\frac{1+y^2}{1-y^2}\right)+\\
&&+\left(\frac{(1-y^2)^2}{4y^2}
   \left[\frac{(n-2)^2}{4}+2w\right]+\frac{n(n-2)}{4}\right)\psi\left(\frac{1+y^2}{1-y^2}\right)\/.
\end{eqnarray*}
Moreover, putting $z=\frac{1+y^2}{1-y^2}$ and using the equality $1-\left(\frac{1+y^2}{1-y^2}\right)^2 = \frac{-4y^2}{(1-y^2)^2}$ lead to the following differential equation for $\psi$
\begin{eqnarray*}
   (1-z^2)\psi''(z)-2z\psi'(z)+\(\nu(\nu+1)-\frac{A(-w)^2}{1-z^2}\)\psi(z)=0\/,\quad z\geq 1\/.
\end{eqnarray*}
with $\nu=\frac{n}{2}-1$ and $A(-w)=\frac{1}{2}\sqrt{(n-2)^2+8w}$. This is the Legendre's differential equation (\ref{Legendre:equation}). Thus, the general solution of (\ref{HBall:Schrodinger:eq}) is given by
\begin{equation*}
   \varphi(y) = c_1 \cdot y^{1-n/2}P_\nu^{-A(-w)}\left(\frac{1+y^2}{1-y^2}\right)+c_2\cdot y^{1-n/2}Q_\nu^{-A(-w)}\left(\frac{1+y^2}{1-y^2}\right)\/,\quad y\in[0,r]\/,
\end{equation*}
where $c_1$ and $c_2$ are absolute constants. Using (\ref{Legendre:first:formula}) and (\ref{Legendre:second:formula}) one can easily check that the function $y^{-\nu}P_\nu^{-A(-w)}\(\frac{1+y^2}{1-y^2}\)$ is bounded on the interval  $\left[1,\frac{1+r^2}{1-r^2}\right)$ in contrast to the function $y^{-\nu}Q_\nu^{-A(-w)}\(\frac{1+y^2}{1-y^2}\)$, which is unbounded in the neighborhood of $1$. Thus $c_2=0$ and the boundary condition $\varphi(r)=1$ gives
\begin{eqnarray}
   \label{HBall:Laplace:formula}
   \mathcal{L}\mu_{|x|}(w) = \left(\frac{r}{|x|}\right)^{n/2-1}\frac{P_\nu^{-A(-w)}\(\frac{1+|x|^2}{1-|x|^2}\)}{P_\nu^{-A(-w)}\(\frac{1+r^2}{1-r^2}\)}\/,\quad |x|\leq r\/,w\geq 0\/.
\end{eqnarray}
Now observe that for every complex number $w$ such that $\textnormal{Re}(w)>-\frac{\nu^2}{2} = -\frac{(n-2)^2}{8}$
 \begin{eqnarray*}
    |\mathcal{L}\mu_{|x|}(w)| &\leq& \E^{|x|}\[\exp\left(-\frac{n(n-2)}{2}\int_0^{\tau_r} \frac{ds}{(1-(R_s^{(\nu)})^2)^2}-\textnormal{Re}(w)\int_0^{\tau_r} \frac{ds}{(R_s^{(\nu)})^2}\right)\]\\
     &\leq& \E^{|x|}\exp\left(\frac{(n-2)^2}{8}\int_0^{\tau_r} \frac{ds}{(R_s^{(\nu)})^2}\right) = \E^{|x|}\exp\left(\frac{\nu^2}{2}\int_0^{\tau_r} \frac{ds}{(R_s^{(\nu)})^2}\right) = \left(\frac{r}{x}\right)^\nu\/.
 \end{eqnarray*}
 The last equality follows from (see \cite{BorodinSalminen:2002} 2.20.4 p.407)
 \begin{eqnarray*}
    \textbf{P}_x^{(\nu)}\left(\int_0^{\tau_r}\frac{ds}{(R^{(\nu)}(s))^2}\in dy\right) = \left(\frac{r}{x}\right)^\nu\frac{\ln(r/x)}{\sqrt{2\pi}y^{3/2}}\exp\left(-\frac{\nu^2 y}{2}-\frac{\ln^2(r/x)}{2y}\right)dy\/,\quad y>0\/.
 \end{eqnarray*}
 In particular $\mathcal{L}\mu_{|x|}(-\nu^2/2)$ is finite. This implies that the formula
 \begin{eqnarray*}
    \mathcal{L}\mu_{|x|}(w) = \int_0^\infty e^{-w t}\mu_{|x|}(dt)
 \end{eqnarray*}
 defines a holomorphic function in the complex half-plane $\textnormal{Re}(w)> - v^2/2$. Moreover, for $|z|<1$ the function $\,_2F_1(\alpha,\beta;\gamma;z)/\Gamma(\alpha)$ as a function of $\alpha$ is analytic function on $\C$. Using this fact and the representation of $P_{\nu}^\mu$ in terms of hypergeometric function $\,_2F_1$ 
  we deduce that the function on the right-hand side of (\ref{HBall:Laplace:formula}) is a meromorphic function in the half-plane $\textnormal{Re}(w)>-\nu^2/2$. In fact, the equality (\ref{HBall:Laplace:formula}) implies that the ratio of Legendre functions is analytic for $\textnormal{Re}(w)>-\nu^2/2$ and consequently the function in the denominator has no zeros in this region. Compare this result with Conjecture 5.2 in \cite{BM:2007}. Moreover, we have just proved that (\ref{HBall:Laplace:formula}) holds whenever $\textnormal{Re}(w)>-\nu^2/2$.

Now let $c=-\frac{(n-2)^2}{16}$. We have
\begin{eqnarray*}
    \int_0^\infty p_t^{S}\(1,\frac{\<x,y\>}{|x||y|}\)\mu_{|x|}(dt)
   &=&\int_0^\infty p_t^{S}\(1,\frac{\<x,y\>}{|x||y|}\)\(\frac{1}{2\pi i}\int_{c-i\infty}^{c+i\infty }e^{zt}\mathcal{L}\mu_{|x|}(z)dz\)dt\\
   &=& \frac{1}{2\pi i}\int_{c-i\infty}^{c+i\infty }\mathcal{L}\mu_{|x|}(z)\(\int_0^\infty e^{zt}p_t^{S}\(1,\frac{\<x,y\>}{|x||y|}\)dt\)dz\\
   &=&  \frac{1}{2\pi i}\int_{c-i\infty}^{c+i\infty }\mathcal{L}\mu_{|x|}(z) G_{z}\(1,\frac{\<x,y\>}{|x||y|}\)dtdz\/.
\end{eqnarray*}
Taking into account the previously found formulas for the Laplace transform $\mathcal{L}\mu_{|x|}(z)$ and the Green function $G_{z}(1,x)$ we finally obtain 
\begin{thm}
For every $x\in B_r$, $x\neq 0$ and $y\in S_r^{n-1}$ the Poisson kernel $P_r(x,y)$ is given by the following formula
\begin{equation*}
    \frac{2\Gamma\(\frac{n+1}{2}\)}{\pi^{\frac{n-1}{2}} n r^{n-1} }\left(\frac{1-|x|^2}{1-r^2}\frac{r}{|x|}\right)^{\nu}\frac{\sin^{\frac{3-n}{2}}\varphi}{2\pi i}\int_{c-i\infty}^{c+i\infty} \frac{P_\nu^{-A(-z)}\(\frac{1+|x|^2}{1-|x|^2}\)}{P_\nu^{-A(-z)}\(\frac{1+r^2}{1-r^2}\)} B_n(z)P_{A(z)-\frac12}^{\frac{3-n}{2}}(-\cos\varphi)dz\/,
\end{equation*}
where  $A(z) = \frac{1}{2}\sqrt{(n-2)^2-8z}$,
$B_{n}(z)  = \frac{\Gamma\(\frac{n-2}{2}-A(z)\)\Gamma\(\frac{n-2}{2}+A(z)\)}{2^{\frac{n+1}{2}}z \Gamma\(\frac{n-1}{2}\)}$, $c=-\nu^2/4$ and $\varphi$ is an angle between $x$ and $y$,

\end{thm}

\section{Harmonic measure of Ornstein-Uhlenbeck process}
As in the previous chapter for fixed $r>0$ we denote $B_r=\{x\in\R^n: |x|<r\}$. Consider a vector field $F_3(x)=\lambda x$, where $\lambda>0$. $F_3$ is a potential vector field with the potential function $V_3(x)={\lambda |x|^2}/{2}$
and as in the previous cases it is orthogonal to the boundary of $B_r$. The corresponding martingale $M$ and its quadratic variation are
\begin{eqnarray*}
M(t)=\lambda \sum_{i=1}^n \int_0^t W_i(s) dW_i(s)\/,\quad \left\langle M\right\rangle(t)=\int_0^t\lambda ^2 |W(s)|^2 ds.
\end{eqnarray*}
The validity of \pref{Girsanovzalozenie} in this case follows from
\begin{eqnarray*}
\E^x[\exp\left\langle M\right\rangle(t\wedge\tau_r)]=\E^x \left[\exp\int_0^{t\wedge\tau_r}\lambda ^2 |W(s)|^2 ds\right]
\le \E^x[\exp [\lambda^2 r^2(t\wedge\tau_r)]]<\infty,
\end{eqnarray*}
where $\tau_r=\inf\{t>0:W(t)\notin B_r\}$. Since
\begin{eqnarray*}
N(t)=\exp{\frac{\lambda(|W(t)|^2-|W(0)|^2)}{2}}
\left[\exp\left(-\frac{1}{2}\int_0^\tau (\lambda^2|W(s)|^2+2n\lambda )ds\right)\right]\/,
\end{eqnarray*}
all the assumptions  of Theorem \ref{Girsanov} are satisfied and consequently we obtain that the harmonic measure $w^x_r(dy)$ of $B_r$ for the operator
$$L_3=\frac{1}{2}\Delta+\lambda x\cdot \nabla$$
is given by
\begin{eqnarray*}
w^x_r(dy)=\exp{\frac{\lambda(r^2-|x|^2)}{2}}\E^x\left[\exp\left(-\frac{1}{2}\int_0^{\tau_r} (\lambda^2|W(s)|^2+2n\lambda )ds\right);\/W(\tau_r)\in dy\right].
\end{eqnarray*}
Computations in this case mimic those introduced in the previous section so we omit some details and present only a sketch of the argumentation. For $x\neq 0$ the skew-product representation of the Brownian motion allows us to write
\begin{equation*}
   \omega_r^x(dy) = e^{\frac{\lambda}{2}(r^2-|x|^2)} \int_0^\infty P^{\frac{x}{|x|}}(\Theta_t \in dy)\E^{|x|}\left[\exp\(\int_0^{\tau_r} q(R^{(\nu)}(s))ds\);\int_0^{\tau_r} \frac{ds}{(R^{(\nu)}(s))^2}\in dt\right]\/,
\end{equation*}
where $q(y) = -\frac{\lambda^2}{2}|x|^2-n\lambda$ and consequently  the Poisson kernel $P_r(x,y)$ is given by
\begin{eqnarray*}
   P_r(x,y) = \frac{2\Gamma\(\frac{n+1}{2}\)}{\pi^{(n-1)/2} n r^{n-1} }e^{\frac{\lambda}{2}(r^2-|x|^2)}\int_0^\infty p_t^{S}\(1,\frac{\<x,y\>}{|x||y|}\)\mu_{|x|}(dt)\/,
\end{eqnarray*}
where
\begin{eqnarray*}
\mu_y(dt) = \E^{y}\left[\exp\(\int_0^{\tau_r} q(R_s^{(\nu)})ds\);\int_0^{\tau_r} \frac{ds}{(R_s^{(\nu)})^2}\in dt\right]\/,\quad y\in (0,r]\/.
\end{eqnarray*}
As previously, the Laplace transform $\mathcal{L}\mu_y(w)$ given by 
\begin{eqnarray}
\label{laplaceOU}
\E^{y}\exp\int_0^{\tau_r}\left( -\frac{\lambda^2}{2}(R_s^{(\nu)})^2-n\lambda-\frac{w}{(R_s^{(\nu)})^2}\right)ds
 = \E^{y}e_g(\tau_r)\/,
\end{eqnarray}
where $g(y) = -\frac{\lambda^2}{2}y^2-n\lambda-\frac{w}{y^2}$, can be identified (by applying the Feynman-Kac formula) as a bounded solution of the following Schr\"o{}dinger equation
\begin{eqnarray*}
   \frac{1}{2}\varphi''(y)+\frac{n-1}{2y}\varphi'(y)-\left(\frac{\lambda^2}{2}y^2+n\lambda+\frac{w}{y^2}\right)\varphi(y)=0\/,\quad y\in (0,r)\/, w\geq 0\/,
\end{eqnarray*}
with the boundary condition $\varphi(r)=1$. Setting $\varphi(y)=y^{-n/2}f(\lambda y^2)$ we reduce this equation to the Whittaker equation
\begin{eqnarray*}
f''(x)+f(x)\left[-\frac{1}{4}-\frac{n}{2x}-\left(\frac{n(n-4)}{16}+\frac{w}{2}\right)\frac{1}{x^2}\right]=0,
\end{eqnarray*}
with parameters $k=-\frac{n}{2}$ and $\mu=\frac{\sqrt{(n-2)^2+8w}}{4}$. Consequently
\begin{eqnarray*}
\varphi(y)=y^{-\frac{n}{2}}\left[c_1 M(k,\mu,\lambda y^2)+c_2W(k,\mu,\lambda y^2)\right],
\end{eqnarray*}
where $M$ and $W$ are Whittaker functions (see \cite{AbramowitzStegun:1972} 13.1.32, 13.1.33 p.505). The boundedness of $\varphi$ implies that $c_2=0$ and the boundary condition $\varphi(r)=1$ gives
\begin{eqnarray}
  \label{eq:OU:laplace}
  \mathcal{L}\mu_y(w)=\left(\frac{r}{y}\right)^{\frac{n}{2}}\frac{M(k,\mu,\lambda y^2)}{M(k,\mu,\lambda r^2)}.
\end{eqnarray}  
If we look at (\ref{laplaceOU}), the probabilistic definition of $\mathcal{L}\mu_y(w)$, by the same argument as previously gives that $\mathcal{L}\mu_y(w)$ can be extended to holomorphic function in the complex half-plane $\textnormal{Re} (w)>-\frac{(n-2)^2}{8}$. Since the Whittaker functions are well defined in this region the formula (\ref{eq:OU:laplace}) is also satisfied in this region. As before we use the Laplace inverse formula and for $c=-\frac{(n-2)^2}{16}$ obtain
\begin{eqnarray*}
    \int_0^\infty p_t^{S}\(1,\frac{\<x,y\>}{|x||y|}\)\mu_{|x|}(dt)
   &=&\int_0^\infty p_t^{S}\(1,\frac{\<x,y\>}{|x||y|}\)\(\frac{1}{2\pi i}\int_{c-i\infty}^{c+i\infty }e^{zt}\mathcal{L}\mu_{|x|}(z)dz\)dt\\
   &=& \frac{1}{2\pi i}\int_{c-i\infty}^{c+i\infty }\mathcal{L}\mu_{|x|}(z)\(\int_0^\infty e^{zt}p_t^{S}\(1,\frac{\<x,y\>}{|x||y|}\)dt\)dz\\
   &=&  \frac{1}{2\pi i}\int_{c-i\infty}^{c+i\infty }\mathcal{L}\mu_{|x|}(z) G_{z}\(1,\frac{\<x,y\>}{|x||y|}\)dtdz\/.
\end{eqnarray*}
\begin{thm}
For every $x\in B_r$, $x\neq 0$ and $y\in S_r^{n-1}$ the Poisson kernel $P_r(x,y)$ is given by the following formula
\begin{equation*}
    \frac{2\Gamma\(\frac{n+1}{2}\)}{\pi^{\frac{n-1}{2}} n r^{n-1} }\left(\frac{1-|x|^2}{1-r^2}\right)^{\frac{n-2}{2}}\left(\frac{r}{|x|}\right)^{\frac{n}{2}}\frac{\sin^{\frac{3-n}{2}}\varphi}{2\pi i}\int_{c-i\infty}^{c+i\infty} \frac{M(k,\mu,\lambda y^2)}{M(k,\mu,\lambda r^2)} B_n(z)P_{A(z)-\frac12}^{\frac{3-n}{2}}(-\cos\varphi)dz\/,
\end{equation*}
where $\varphi$ is an angle between $x$ and $y$, $A(z) = \frac{1}{2}\sqrt{(n-2)^2-8z}$ and
\begin{eqnarray*}
B_{n}(z)  = \frac{\Gamma\(\frac{n-2}{2}-A(z)\)\Gamma\(\frac{n-2}{2}+A(z)\)}{2^{\frac{n+1}{2}}z \Gamma\(\frac{n-1}{2}\)}\/.
\end{eqnarray*}
\end{thm}
\section{Appendix}
For convenience of the Reader we collect here basic information about Bessel functions, hypergeometric functions and other special functions appearing throughout the paper. Mainly we follow the exposition given in \cite{AbramowitzStegun:1972} and \cite{Erdelyi:1955}, where we refer the Reader for more details (see also \cite{Lebedev:1972} and \cite{GradsteinRyzhik:2007}). 

\subsection{Bessel functions}
The Bessel functions $J_\nu(z)$ and $Y_\nu(z)$ are independent solutions of the Bessel equation
\begin{eqnarray*}
   z^2 y''(z) + zy'(z) + (z^2-\nu^2)y(z) = 0\/, \quad \nu\in \R\/.
\end{eqnarray*}
The \textit{Wronskian} of the pair $(J_\nu(z),Y_\nu(z))$ is equal to $2/{(\pi z)}$ (see \cite{Lebedev:1972}, p.113). 
The derivatives of Bessel functions can be expressed by itself in the following way
\begin{eqnarray}
\label{Jderivative}
J_\nu'(x)&=&J_{\nu-1}(x)-\frac{\nu}{x}J_\nu(x)\/,\quad x>0\/,\\
\label{Yderivative}
Y_\nu'(x)&=&Y_{\nu-1}(x)-\frac{\nu}{x}Y_\nu(x)\/,\quad x>0.
\end{eqnarray}
For every $\nu>0$ we have (see \cite{Lebedev:1972} 5.16 p.134-135).
\begin{eqnarray}
\label{eq:mBessel:asymptotyka0}
J_\nu(x)\sim\frac{x^\nu}{2^\nu \Gamma(\nu+1)}&\/,&\quad 
Y_\nu(x)\sim-\frac{2^\nu \Gamma(\nu)}{\pi}\frac1{x^\nu}\/,\quad x\to 0^{+}\/.\\
\label{eq:mBessel:asymptotykainfty}
J_\nu(x)\sim\sqrt{\frac{2}{x\pi}}\cos(x-\frac{\nu\pi}{2}-\frac{\pi}{4})&\/,&\quad 
Y_\alpha(x)\sim\sqrt{\frac{2}{x\pi}}\sin(x-\frac{\nu\pi}{2}-\frac{\pi}{4}),\quad x\to\infty\/.
\end{eqnarray}
The modified Bessel functions are independent solutions to the modified Bessel equation
\begin{eqnarray}
\label{eq:mBessel:Eq}
z^2 y''(z) + zy'(z) - (\nu^2+z^2)y(z) = 0\/.
\end{eqnarray}
The following asymptotic expansion for the function $K_\nu$ holds (see \cite{GradsteinRyzhik:2007} 8.451 (6))
\begin{eqnarray}
   \label{eq:mBessel:Asym}
   K_\nu(z) = \sqrt{\frac{\pi}{2z}} \,e^{-z}(1+E(z))\/,\quad |E(z)|=O(|z|^{-1})\textrm{ as }|z|\to \infty\/.
\end{eqnarray}
whenever $|\arg z|\leq 3\pi/2$ . The behavior of $K_\nu$ near zero is described by (see \cite{AbramowitzStegun:1972}, 9.6.9)
\begin{eqnarray}
   \label{eq:mBessel:AsymZero}
   K_\nu(z) \approx \frac{2^{\nu-1}\Gamma(\nu)}{z^{\nu}}\/,\quad \re z> 0\/.
\end{eqnarray}
The connection between modified Bessel function of purely imaginary argument and Bessel functions is given by
\begin{eqnarray}
   \label{eq:mBessel:ImArg}
   K_\nu(ix) = -\frac{i\pi}{2}e^{-i\nu \pi/2}(J_\nu(x)-iY_\nu(x))\/,\quad x>0\/.
\end{eqnarray}
Finally, we recall the integral representation of $K_\nu$ (see \cite{GradsteinRyzhik:2007}, 8.432 (7))
\begin{eqnarray}
\label{eq:mBessel:Integral}
K_\vartheta(z) = \frac{z^\vartheta}{2}\int_0^\infty \exp\(-\frac{t+z^2/t}{2}\)t^{-\vartheta-1}dt\/,\quad z>0\/,\vartheta \in \R\/,
\end{eqnarray}
as well as the formula
(\cite{GradsteinRyzhik:2007} 6.561 formula 16, p. 676)
\begin{eqnarray}\label{eq:mBessel:Integral2}
\int_0^\infty x^\mu K_\nu(x)dx = 2^{\mu-1}\Gamma\left(\frac{1+\mu+\nu}{2}\right)\Gamma\left(\frac{1+\mu-\nu}{2}\right),\quad \mu+1>\nu>0.
\end{eqnarray}
For every $\nu>1/2$ we introduce the following function of two variables
\begin{eqnarray*}
   g_\nu(x,t)=\frac{1}{t^{2\nu}(x-1)}\frac{J_\nu(t)Y_\nu(tx)-J_\nu(tx)Y_\nu(t)}{J^2_\nu(t)+Y_\nu^2(t)}\/,\quad x>1\/,t>0\/.
\end{eqnarray*}
It is obvious that $g_\nu$ is a continuous function on $(1,\infty)\times \R_+$. However, the most crucial for the considerations given in Chapter~\ref{Ch:Hn} are the following asymptotic properties of $g_\nu$.

\begin{lem}
\label{granice}
Set $x_0\ge 1$, $t_0>0$. Then
\begin{eqnarray}
\label{eq:JY:limit1}
\lim_{(x,t)\to(x_0,0)}g_\nu(x,t)&=&\frac{\pi\sum_{k=0}^{n-2}x_0^k}{2^{2\nu}\Gamma(\nu)\Gamma(\nu+1)x_0^\nu},\\
\label{eq:JY:limit2}
\lim_{(x,t)\to(1,t_0)}g_\nu(x,t)&=&\frac{2}{\pi t_0^{2\nu}\left(J_\nu^2(t_0)+Y_\nu^2(t_0)\right)}\/.\label{3}
\end{eqnarray}
\end{lem}
\begin{proof}
\noindent From the Lagrange theorem, there exist $\theta_1,\theta_2\in(1,x)$, depending on $x$ and such that
\begin{eqnarray*}
\frac{J_\nu(t)Y_\nu(tx)-J_\nu(tx)Y_\nu(t)}{tx-t}
&=&J_\nu(t)\frac{Y_\nu(tx)-Y_\nu(t)}{tx-t}-Y_\nu(t)\frac{J_\nu(tx)-J_\nu(t)}{tx-t}\\
&=&J_\nu(t)Y'_\nu(t\theta_1)-J'_\nu(t\theta_2)Y_\nu(t).
\end{eqnarray*}
Obviously, when $x\to 1$ then $\theta_1$, $\theta_2$ also tend to $1$. Furthermore, since the Wronskian of $(J_\nu(z), Y_\nu(z))$ is $2/\pi z$ we get that
\begin{eqnarray*}
\lim_{(x,t)\to(1,t_0)}\frac{J_\nu(t)Y_\nu(tx)-J_\nu(tx)Y_\nu(t)}{tx-t}&=&J_\nu(t_0)Y'_\nu(t_0)-J'_\nu(t_0)Y_\nu(t_0) = \frac{2}{\pi t_0}\/,
\end{eqnarray*}
which proves (\ref{3}).
\noindent If we use the recurrent formulas for the Bessel function derivatives (\ref{Jderivative}) and (\ref{Yderivative}) we obtain that $J_\nu(t)Y'_\nu(t\theta_1)-J'_\nu(t\theta_2)Y_\nu(t)$ is equal to 
\begin{eqnarray*}
J_\nu(t)Y_{\nu-1}(t\theta_1)-\frac{\nu}{t}\left(\frac{J_\nu(t)Y_\nu(t\theta_1)}{\theta_1}-\frac{J_\nu(t\theta_1)Y_\nu(t)}{\theta_2}\right)-J_{\nu-1}(t\theta_2)Y_\nu(t)
\end{eqnarray*}
\noindent Multiplying the last expression by $t$, taking $(x,t)\to(1,0)$ and using (\ref{eq:mBessel:asymptotyka0}) it is easy to see that first two summands tends to zero and the last one tends to $2/\pi$. Since, by (\ref{eq:mBessel:asymptotyka0}), we have $\lim_{t\to 0^+}t^{2\nu}(J^2_\nu(t)+Y^2_\nu(t))=2^{2\nu}\Gamma^2(\nu)/\pi^2$ and consequently
\begin{eqnarray*}
\lim_{(x,t)\to(1,0)}g_\nu(x,t)&=&\frac{\pi}{2^{2\nu-1}\Gamma^2(\nu)}\/,
\end{eqnarray*}
which is (\ref{eq:JY:limit1}) for $x_0=1$. For $x_0>1$ relation (\ref{eq:JY:limit1}) follows directly from (\ref{eq:mBessel:asymptotyka0}) 
\begin{eqnarray*}
   \lim_{(x,t)\to(x_0,0)}g_\nu(x,t) &=& \frac{\pi(x_0^\nu-x_0^{-\nu})}{\nu2^{2\nu}\Gamma^2(\nu)} = \frac{\pi\sum_{k=0}^{n-2}x_0^k}{2^{2\nu}\Gamma(\nu)\Gamma(\nu+1)x_0^\nu}\/.
\end{eqnarray*}
\end{proof}
Note that Lemma~\ref{granice} together with (\ref{eq:mBessel:asymptotyka0}) and (\ref{eq:mBessel:asymptotykainfty}) imply that the function $g_\nu$ can be extended to the continuous function on $[1,\infty)\times[0,\infty)$ which is bounded whenever $x$ is bounded, i.e. for every $R>1$ there exists $C(R)>0$ such that
\begin{eqnarray}
\label{oszacowaniafr}
    |g_\nu(x,t)|\leq C(R)\/,\quad (x,t)\in[1,R]\times[0,\infty)\/.
\end{eqnarray}

\subsection{Hypergeometric Function and Legendre Functions}
For $\gamma \neq -1,-2,\ldots$ the hypergeometric function is defined by
\begin{eqnarray*}
 _2F_1(\alpha,\beta;\gamma;z) = \sum_{k=0}^\infty \frac{(\alpha)_k(\beta_k)}{(\gamma_k)k!}z^k\/,\quad |z|<1\/,
\end{eqnarray*}
Here $(a)_k = \Gamma(a + k)/\Gamma(a)$. Function $_2F_1$ is a solution of the hypergeometric equation 
\begin{eqnarray}
   \label{Hyper:equation}
   z(1-z)u''(z)+[\gamma-(\alpha+\beta+1)z]u'(z)-\alpha\beta u(z) = 0
\end{eqnarray}
regular at $z=0$.  Whenever $\textnormal{Re }(\gamma-\alpha-\beta)>0$ we have (see \cite{Erdelyi:1955} vol.1 p.104 2.8(46))
\begin{eqnarray}
\label{Hyper:value1}
  _2F_1(\alpha,\beta;\gamma;1) = \frac{\Gamma(\gamma)\Gamma(\gamma-\alpha-\beta)}{\Gamma(\gamma-\alpha)\Gamma(\gamma-\beta)}\/.
\end{eqnarray}
The derivative of $_2F_1$ is given by (see \cite{Erdelyi:1955} vol.1 p.102 2.8(20))
\begin{eqnarray}
   \label{Hyper:derivative}
   \dfrac{d}{dz}\,_2F_1(\alpha,\beta;\gamma;z) = \frac{\alpha\beta}{\gamma}\,_2F_1(\alpha+1,\beta+1;\gamma+1;z)\/.
\end{eqnarray}
and the following elementary relation holds (see \cite{Erdelyi:1955} vol.1 p.105 2.9(2))
\begin{eqnarray}
  \label{Hyper:relation}
  \,_2F_1(\alpha,\beta;\gamma;z) = (1-z)^{\gamma-\alpha-\beta}\,_2F_1(\gamma-\alpha,\gamma-\beta;\gamma;z)\/.
\end{eqnarray}
The Legendre functions are solutions of Legendre's differential equation 
\begin{eqnarray}
    \label{Legendre:equation}
   (1-z^2)u''(z)-2zu'(z)+[a(a+1)-b^2(1-z^2)^{-1}]u(z) = 0\/.
\end{eqnarray}
Making appropriate substitution it can be reduced to the hypergeometric equation (\ref{Hyper:equation}) and consequently its solutions are given in terms of hypergeometric function. More precisely, the Legendre function of the first and second kind are defined by (see \cite{Erdelyi:1955} vol.1 p.122 3.2(3) and p.143 3.4(6))
\begin{eqnarray}
   \label{Legendre:first:formula}
   P_a^b(x) &=& \frac{1}{\Gamma(1-b)}\left(\frac{1+x}{|1-x|}\right)^{b/2}\,_2F_1(-a,a+1;1-b;\frac{1-x}{2})\/,\quad x>-1\/.\\
   \label{Legendre:second:formula}
   Q_a^b(x) &=&\frac{ e^{bi\pi}\pi^{1/2}}{2^{a+1}x^{a+b+1}}\frac{\Gamma(a+b+1)}{\Gamma(a+3/2)}(x^2-1)^{b/2}\times\\
   \nonumber
   &&\times \,_2F_1(\frac{a+b}{2}+1,\frac{a+b+1}{2};a+3/2;\frac{1}{x^2})\/,\quad x>1\/,
\end{eqnarray}
respectively. The functions $P_a^b$ and $Q_a^b$ are independent solutions of (\ref{Legendre:equation}).

\subsection{Skew-product representation of Brownian motion}
We now introduce the spherical Brownian motion on the unit sphere ${S}^{n-1}_1 \subseteq \R^{n}$ as a
diffusion on $S^{n-1}_1$ with the generator being the one half of the Laplace-Beltrami operator $\Delta_{S^{n-1}_1}$ 
of the manifold $S^{n-1}_1$. It is well-known that the following formula holds 
\begin{equation*}
\Delta_{S^{n-1}_1} = (\sin\phi)^{2-n}\,\frac{\partial}{\partial\,\phi}[(\sin\phi)^{n-2}\,\frac{\partial}{\partial\,\phi}]+
(\sin\phi)^{-2}\Delta_{S^{n-2}_1}\,,
\end{equation*}
where the angle $\phi$ is the angle between the pole and the given point on the sphere and $\Delta_{S^1_1}=\frac{\partial^2}{\partial\,\phi^2}$. Now, if we consider the action of $\Delta_{S^{n-1}_1}$ on function depending only on $\phi$, 
this reduces to the generator of the Legendre process $LEG(d)$:
\begin{equation*}
\frac12 \Delta_{S^{n-1}_1} = \frac12(\sin\phi)^{2-n}\,\frac{\partial}{\partial\,\phi}[(\sin\phi)^{n-2}\,\frac{\partial}{\partial\,\phi}]=\frac{1}{2}\frac{\partial^2}{\partial\,\phi^2} + \frac{n-2}{2}\,\cot\,\phi\,\frac{\partial}{\partial\,\phi}
\end{equation*}
Changing variable $\cos\,\phi = t$ we obtain
\begin{equation*}
\frac{1-t^2}{2}\frac{\partial^2}{\partial\,t^2} - \frac{n-1}{2}\,t\,\frac{\partial}{\partial\,t}\/.
\end{equation*}
We now invoque the classical skew-product representation of the $n$-dimensional Brownian motion 
(see, e.g. Ito-McKeane 7.15) stating that it can be represented as the product of $R^{(\nu)}=\{R_t^{(\nu)};t\geq 0\}$ - the Bessel process BES($n$), $\nu=n/2-1$ with the generator
\begin{equation}
\label{Bessel:generator}
\frac{1}{2}\frac{\partial^2}{\partial\,r^2} + \frac{n-1}{2r}\,\frac{\partial}{\partial\,r}\,,
\end{equation}
and independent spherical Brownian motion $\Theta=\{\Theta(t);t\geq 0\}$ on $S^{n-1}_1$ with time changed according to the formula
\begin{equation*}
A^{(\nu)}(t) = \int_0^t \frac{ds}{(R^{(\nu)}(s))^2}\/.
\end{equation*}
Moreover, for $x\neq 0$, we introduce the process $S=\{S(t);t\geq 0\}$ defined by $S(t) = \frac{\<x,\Theta(t)\>}{|x|}$.
The process $S$ describes the cosine of the angle between the starting point $x$ and the spherical Brownian motion $\Theta$. Consequently, the cosine between the starting point $x$ and $W(t)$ is just $S(A(t))$.
The skew-product representation and the previous given considerations imply that $S$ is independent from the Bessel process $R^{(\nu)}$ and the generator of $S$ is given by
\begin{eqnarray}
   \label{S:generator}
   \mathcal{G} = \frac{1-t^2}{2}\frac{\partial^2}{\partial\,t^2} - \frac{n-1}{2}\,t\,\frac{\partial}{\partial\,t}
\end{eqnarray}
with a domain $D_\mathcal{G} = \{u\in\mathcal{C}^2[-1,1];\, u'(-1)=u'(1)=0\}$. Three basic characteristics of the diffusion: the speed measure, the scale function and the killing measure are described by the following relations (see also \cite{BMR2:2010}) $m(dx) = 2(1-x^2)^{(n-3)/2}dx$, $s'(x) = (1-x^2)^{(1-n)/2}$, $k(dx)=0$. Moreover, the points $-1$ and $1$ are non-singular reflecting points. We denote by $p_t^S(x,y)$ the transition density function with respect to the speed measure, i.e.
\begin{eqnarray}
   \label{S:ptxy:definition}
   P^x(S(t)\in A) = \int_A p_t^S(x,y)m(dx)\/,\quad A \in \mathcal{B}[-1,1]\/.
\end{eqnarray}

\bibliography{bibliography}
\bibliographystyle{plain}

\end{document}